\documentclass{amsart}

\usepackage{amsmath, amssymb}

\newtheorem{theorem}{Theorem}
\newtheorem{lemma}[theorem]{Lemma}

\newcommand{\be}{\begin{equation}}
\newcommand{\ee}{\end{equation}}

\newcommand{\uv}{\mathbf{u}}
\newcommand{\sigmav}{\boldsymbol{\sigma}}
\newcommand{\tauv}{\boldsymbol{\tau}}
\newcommand{\nv}{\mathbf{n}}
\newcommand{\vv}{\mathbf{v}}
\newcommand{\fv}{\mathbf{f}}
\newcommand{\zv}{\mathbf{z}}
\newcommand{\xv}{\mathbf{x}}
\newcommand{\tv}{\mathbf{t}}

\newcommand{\sigmat}{\underline{\boldsymbol{\sigma}}}
\newcommand{\taut}{\underline{\boldsymbol{\tau}}}
\newcommand{\epst}{\underline{\boldsymbol{\varepsilon}}}

\newcommand{\Atel}{{\underline{\mathbf{C}}^{-1}}}
\newcommand{\Ctel}{\underline{\mathbf{C}}}
\newcommand{\Ft}{\underline{\mathbf{F}}}
\newcommand{\St}{\underline{\mathbf{S}}}

\newcommand{\tang}{{\mathbf{t}}}
\newcommand{\triang}{\mathcal{T}}
\newcommand{\faces}{\mathcal{F}}

\newcommand{\GammaD}{\Gamma_D}
\newcommand{\GammaN}{\Gamma_N}

\newcommand{\spaceSigma}{\underline{\mathbf{\Sigma}}}
\newcommand{\spaceW}{W}
\newcommand{\spaceV}{\mathbf{V}}

\newcommand{\opB}{B}
\newcommand{\bilA}{a}
\newcommand{\bilB}{b}

\newcommand{\opdiv}{\operatorname{div}}
\newcommand{\opdivv}{\mathbf{{div}}}
\newcommand{\opcurl}{\mathbf{curl}}
\newcommand{\opker}{\operatorname{Ker}}
\newcommand{\nablat}{\underline{\nabla}}

\newcommand{\IntV}{\mathcal{I}_{\spaceV}}
\newcommand{\IntW}{\mathcal{I}_{\spaceW}}
\newcommand{\IntSigma}{\mathcal{I}_{\spaceSigma}}
\newcommand{\Clement}{\mathcal{C}}

\newcommand{\Lv}{{\mathbf{L}}}
\newcommand{\Hv}{{\mathbf{H}}}
\newcommand{\Lt}{{\underline{\mathbf{L}}}}
\newcommand{\Ct}{{\underline{\mathbf{C}}}}
\newcommand{\Ht}{{\underline{\mathbf{H}}}}
\newcommand{\Pv}{{\mathbf{P}}}
\newcommand{\Pt}{{\underline{\mathbf{P}}}}

\begin{document}


\title{An analysis of the TDNNS method using natural norms}

\author{Astrid S.\ Pechstein}
\address{Astrid. S.\ Pechstein\\ Institute of Technical Mechanics\\ Johannes Kepler University Linz\\ Altenbergerstr. 69\\ 4040 Linz, Austria}
\email{astrid.pechstein@jku.at}
\author{Joachim\ Sch\"oberl}
\address{Joachim\ Sch\"oberl\\ Institute for Analysis and Scientific Computing\\ Vienna University of Technology\\ Wiedner Hauptstrasse 8-10\\ 1040 Wien, Austria}
\email{joachim.schoeberl@tuwien.ac.at}


\date{\today}

\begin{abstract}
The Tangential-Displacement Normal-Normal-Stress (TDNNS) method is a finite element method for mixed elasticity. As the name suggests, the tangential component of the displacement vector as well as the normal-normal component of the stress are the degrees of freedom of the finite elements. The TDNNS method was shown to converge of optimal order, and to be robust with respect to shear and volume locking. However, the method is slightly nonconforming, and an analysis with respect to the natural norms of the arising spaces was still missing.
We present a sound mathematical theory of the infinite dimensional problem using the space $\Hv(\opcurl)$ for the displacement. We define the space $\Ht(\opdiv\opdivv)$ for the stresses and provide trace operators for the normal-normal stress. Moreover, the finite element problem is shown to be stable with respect to the $\Hv(\opcurl)$ and a discrete $\Ht(\opdiv\opdivv)$ norm. A-priori error estimates of optimal order with respect to these norms are obtained.


\end{abstract}

\maketitle

\section{Introduction} \label{sec:intro}

In \cite{PechsteinSchoeberl:11}, we introduced the TDNNS method for treating the problem of linear elasticity. The TDNNS method is a finite element method that uses tangential-continuous elements for the displacements and symmetric normal-normal continuous finite elements for the stresses. We showed that the TDNNS method is capable of overcoming shear locking \cite{PechsteinSchoeberl_anis:11} and volume locking \cite{Sinwel:09}.

However, the TDNNS method is slightly nonconforming, as the stress finite elements are not in the infinite-dimensional distributional space $\Ht(\opdiv\opdivv)$, which was introduced in \cite{PechsteinSchoeberl:11}. The analysis of TDNNS finite elements provided in our former work \cite{PechsteinSchoeberl:11,PechsteinSchoeberl_anis:11} is based on discrete, broken norms rather than the natural norms of the infinite-dimensional spaces $\Hv(\opcurl)$ and $\Ht(\opdiv\opdivv)$. In the present paper, we want to provide an analysis based on the natural norms of the Sobolev spaces. This analysis takes the fact that the stress space is nonconforming into account, and leads to optimal order a-priori error estimates.

\subsection{Notation}
We shortly present the notation used throughout the paper: Vectors shall be denoted as boldface (e.g.\ $\uv$), while tensors are boldface and underlined (e.g.\ $\sigmat$). On the boundary of some domain $A$, we use the outer normal vector $\nv$. For a vector field $\uv$, $u_n = \uv \cdot \nv$ is the normal component, and $\uv_\tang = \uv - u_n \nv$ is the tangential component. For a tensor field $\sigmat$, let $\sigmav_{\nv} = \sigmat \nv$ be the normal component, which is further split into its normal-normal component $\sigma_{nn} = (\sigmat \nv) \cdot \nv$ and its normal-tangential component $\sigmav_{\nv\tang} = \sigmav_\nv - \sigma_{nn} \nv$. 

Gradient, curl and divergence operators $\nabla$, $\opcurl$ and $\opdiv$ operators are defined in the usual way.
The gradient $\nablat$ of a vector field is a tensor containing in each row the gradient of the corresponding vector component. 
The divergence $\opdivv$ of a tensor field is a vector, where each component is the divergence of the corresponding row of the tensor. 

On some domain $A$, we use the Lebesgue space $L^2(A)$ and the standard Sobolev space $H^1(A)$ of weakly differentiable $L^2$ functions with gradient in $L^2(A)$. To indicate vector or tensor valued spaces, we use $\Lv^2(A)$, $\Hv^1(A)$ and $\Lt^2(A), \Ht^1(A)$, respectively. The space of tensor-valued symmetric functions with components in $L^2(A)$ is denoted as $\Lt^2_{sym}(A)$. The space of smooth functions on the closure $\bar A$ is denoted as $C^\infty(\bar A)$, and $C^\infty_{0,\Gamma}(\bar A)$ is the subspace where all derivatives vanish on the boundary part $\Gamma$ of $A$. If the domain of interest $\Omega$ is concerned, it may be omitted, writing e.g.\ $H^1$ for $H^1(\Omega)$.

On the boundary of a domain we use differential operators and spaces as introduced in the work of Buffa and Ciarlet \cite{BuffaCiarlet:01a,BuffaCiarlet:01b}. For the exact definitions, we refer to their work. We will mostly need the rather well-known trace space $H^{1/2}(\partial A)$ and the spaces $H^{1/2}(\Gamma), H^{1/2}_{00}(\Gamma)$ on a part $\Gamma$ of the boundary, where the latter can be extended by zero to the whole boundary space $H^{1/2}(\partial A)$.

\subsection{Problem geometry} \label{sec:domain}
Throughout the paper, we assume the domain of interest $\Omega \subset \mathbb R^3$ to be a bounded, connected, polyhedral domain with Lipschitz boundary $\partial \Omega$. Note that all results can directly be transferred to the two-dimensional case.

The (closed) polygonal faces of the polyhedral domain $\Omega$ shall be denoted by $\Gamma_i$ with $i \in \mathcal{I}$ and $\mathcal{I}$ a suitable index set. Different boundary conditions will be prescribed different parts of the boundary $\partial \Omega$. To this end, the boundary is divided into two closed parts $\Gamma_D$ and $\Gamma_N = \partial \Omega \backslash \mathrm{int}(\Gamma_D)$. The boundary part $\Gamma_D$, where the displacement will be prescribed (Dirichlet boundary condition), shall be non-trivial, whereas the boundary part $\Gamma_N$, where surface tractions are given (Neumann boundary condition), may vanish. 

We assume that both $\Gamma_D$ and $\Gamma_N$ are aligned with the boundary faces $\Gamma_i$, such that they each are a union of boundary faces,
\begin{equation}
	\Gamma_D = \bigcup_{i \in \mathcal{I}_D} \Gamma_i,\qquad \Gamma_N = \bigcup_{i \in \mathcal{I}_N} \Gamma_i.
\end{equation} 

In accordance with \cite{HiptmairZheng:09} we assume that for each connected component of the Dirichlet boundary $\Gamma_{D,i}$ we can find an open Lipschitz domain $\Omega_{D,i} \subset \mathbb R^3$ such that
\begin{equation}
	\overline \Omega_{D,i} \cap \overline \Omega = \Gamma_{D,i},\qquad \Omega_{D,i} \cap \Omega = \emptyset.
\end{equation} 
Moreover, $\Omega_{D,i}$ and $\Omega_{D,j}$ have positive distance for $i\not = j$, and the interior of $\overline \Omega \cup \bigcup \overline \Omega_{D,i}$ is Lipschitz.

\subsection{Linear elasticity}
Let $\uv: \Omega \to \mathbb R^3$ be the displacement vector. In linear elasticity, we use the linearized strain tensor
\begin{equation}
	\epst(\uv) = \frac12 \left( \nablat \uv + \nablat \uv^T\right).
\end{equation}
We are interested in finding displacement vector $\uv$ and symmetric stress tensor $\sigmat$ satisfying
\begin{align}
	\Atel \sigmat &= \epst(\uv) & \mathrm{in}\ \Omega, \label{eq:Hooke}\\
	-\opdivv \sigmat &= \fv & \mathrm{in}\ \Omega. \label{eq:equilibrium}
\end{align}
Hooke's law \eqref{eq:Hooke}  connects strain and stress tensor by the compliance tensor  $\Atel$, which is the inverse of the standard elasticity tensor $\Ctel$ depending on Young's modulus $E$ and the Poisson ratio $\nu$. We assume that Young's modulus $E$ is bounded, and the Poisson ratio $\nu$ is bounded away from 1/2, such that both $\Ctel$ and $\Atel$ exist and lie in $\Lt^{\infty}(\Omega)$. Equation \eqref{eq:equilibrium} is the equilibrium condition. 

We assume that all boundary conditions are prescribed on the boundary parts $\Gamma_D$ and $\Gamma_N$ introduced above. The displacement shall be given on $\Gamma_D$, while tractions are given on $\Gamma_N$,
\begin{align}
	\uv &= \uv_D & & \mathrm{on}\ \GammaD, \label{eq:boundarycondD} \\
	\sigmav_\nv &= \tv_N & & \mathrm{on}\ \GammaN \label{eq:boundarycondN}.
\end{align}

\subsection{Motivation of the TDNNS method}

Two different variational formulations are widely known for the partial differential equations  \eqref{eq:Hooke}, \eqref{eq:equilibrium}. Most standard finite element methods rely on a primal formulation, where the stress tensor $\sigmat$ is eliminated. In this formulation, the displacement boundary condition on $\Gamma_D$ is essential, and usually treated by a homogenization approach. To this end, it is necessary to have the existence of an extension $\tilde \uv_D \in \Hv^1(\Omega)$  of the boundary displacement $\uv_D$ to the whole domain $\Omega$. Then one searches for $\uv \in \tilde \uv_D + \Hv^1_{0,\GammaD}(\Omega)$ with the space $\Hv^1_{0,\GammaD}(\Omega) = \{ \vv \in \Hv^1(\Omega): \vv = 0\ \mbox{on}\ \GammaD\}$ satisfying the homogeneous displacement boundary condition,
\begin{equation}
	\int_\Omega \Ctel \epst(\uv) : \epst(\vv) \, d\xv = \int_\Omega \fv \cdot \vv\, d\xv + \int_{\Gamma_N} \tv_N \cdot \vv\,ds\qquad \forall \vv \in \Hv^1_{0,\GammaD}(\Omega). \label{eq:variationalH1}
\end{equation}
To be conforming, the displacement finite element space has to be continuous.

On the other hand, a dual Hellinger-Reissner formulation can be obtained from system \eqref{eq:Hooke}, \eqref{eq:equilibrium}. 
Integration by parts puts all continuity assumptions to the stress tensor. It has to allow for a weak divergence, while only $\Lv^2$ regularity is required for $\uv$. In this case, the traction boundary conditions are essential. One needs an extension of the surface tractions $\tv_N$ to the domain, a tensor field $\tilde \sigmat_N \in \Ht_{sym}(\opdivv)$ with $\tilde \sigmav_{N,\nv} = \tv_N$ on $\Gamma_N$. Inhomogeneous displacement boundary conditions can be included in weak form into the right hand side of equation \eqref{eq:variationalHdivI}.

One searches for $\sigmat \in \tilde \sigmat_N + \Ht_{sym,0,\GammaN}(\opdivv)$ with the space $ \Ht^{sym}_{0,\GammaN}(\opdivv) = \{\sigmat \in \Lt^2_{sym}(\Omega): \opdivv \sigmat \in \Lv^2(\Omega), \sigmav_\nv = 0 \ \mbox{on}\ \GammaN\}$ satisfying the homogeneous traction boundary condition and $\uv \in \Lv^2(\Omega)$ such that
\begin{align}
	\int_\Omega \Atel \sigmat : \taut\, d\xv + \int_\Omega \opdivv \taut \cdot \uv\, d\xv &= \int_{\Gamma_D} \uv_D \cdot \tauv_\nv\,ds && \forall \taut \in \Ht^{sym}_{0,\GammaN}(\opdivv), \label{eq:variationalHdivI}\\
	\int_\Omega \opdivv(\sigmat) \cdot \vv\, d\xv &= -\int_\Omega \fv \cdot \vv\, d\xv && \forall \vv \in \Lv^2(\Omega). \label{eq:variationalHdivII}
\end{align}
To define a conforming finite element method, one has to provide stress elements which are symmetric and normal continuous. Such elements have been found \cite{ArnoldWinther:02,AdamsCockburn:05,ArnoldAwanouWinther:08}, but come only at high computational costs, as they involve at least 24 degrees of freedom per element in two dimensions or 162 in three dimensions.

The TDNNS formulation is in between the primal and the dual concept. We want to design a formulation, where the tangential component $\uv_\tang$ of the displacement and the normal component $\sigma_{nn}$ of the normal stress vector are essential boundary conditions on the respective boundary parts $\Gamma_D$ and $\Gamma_N$. In other words, the displacement space has to allow for the definition of a tangential trace, while the stress space allows a normal-normal trace. It will turn out that the displacement space is the space $\Hv_{0,\Gamma_D}(\opcurl)$ satisfying zero tangential boundary conditions on $\Gamma_D$. 

Below, we formally write the variational formulation. It is of the standard mixed form treated in \cite{BoffiBrezziFortin:13}. We use the stress space $\spaceSigma$ and displacement space $\spaceV$, which will be rigorously defined in Section~\ref{sec:variational}. Currently, we only state that $\vv \in \spaceV$ implies $\vv_\tang = 0$ on $\Gamma_D$ and $\taut \in \spaceSigma$ implies $\tau_{nn} = 0$ on $\Gamma_N$. Accordingly, we need two extensions, one for the tangential component of the displacement and one for the normal-normal component of the stress: on the Dirichlet boundary $\Gamma_D$ some $\tilde \uv_{D}$ with $\tilde \uv_{D,\tang} = \uv_{D,\tang}$  and on the Neumann boundary $\Gamma_N$ some $\tilde \sigmat_N$ with $\tilde \sigma_{N,nn} = t_{N,n}$. We want to find $\uv$ in $\tilde \uv_D + \spaceV$ and $\sigmat \in \tilde \sigmat_N + \spaceSigma$ such that
\begin{align}
	\bilA(\sigmat,\taut) + \bilB(\taut, \uv) &= \int_{\Gamma_D} u_{D,n} \tau_{nn}\,ds && \forall \taut \in \spaceSigma, \label{eq:spp1}\\
	\bilB(\sigmat, \vv) &= -\int_\Omega \fv \cdot \vv\, d\xv - \int_{\Gamma_N} \tv_{N,\tang} \cdot \vv_\tang\,ds&& \forall \vv \in \spaceV \label{eq:spp2}.
\end{align}
For smooth functions, the bilinear forms $\bilA(\cdot,\cdot)$ and $\bilB(\cdot,\cdot)$ are given by
\begin{align}
	\bilA(\sigmat,\taut) &= \int_\Omega \Atel \sigmat : \taut\, d\xv, \label{eq:defineA} \\
	\bilB(\taut, \vv)    &= - \int_{\Omega} \epst(\vv) : \taut\, d\xv + \int_{\partial \Omega} \tau_{nn} v_n\, ds. \label{eq:defineBcontI} \\
		&= \int_{\Omega} \opdivv\taut \cdot vv\, d\xv - \int_{\partial \Omega} \tauv_{\nv\tang}\vv_\tang\, ds. \label{eq:defineBcontII}
\end{align}
In Section~\ref{sec:variational} we will define the function spaces and give a precise meaning to the arising integrals and bilinear forms in a distributional setting. We will determine in which way the boundary terms have to be understood. We will see that the bilinear form $\bilB(\cdot,\cdot)$ corresponds to a distributional divergence operator.

\section{The variational formulation of the TDNNS method} \label{sec:variational}

We shall specify the spaces, in which the variational formulation \eqref{eq:spp1} - \eqref{eq:spp2} is posed. While the displacement space is well-known, the stress space was introduced in \cite{PechsteinSchoeberl:11} and shall be analysed in detail in this work. 

\subsection{The displacement space $\Hv(\opcurl)$}

We use the space
\begin{equation}
	\Hv(\opcurl) = \left\{ \vv \in \Lv^2(\Omega): \opcurl \vv \in \Lv^2(\Omega)\right\}.
\end{equation}
This is a Hilbert space equipped with inner product and induced norm
\begin{equation}
	(\uv,\vv)_{\Hv(\opcurl)} = \int_{\Omega} (\uv \cdot \vv+ \opcurl \uv \cdot \opcurl \vv)\, d\xv,\ \|\uv\|_{\Hv(\opcurl)}^2 = (\uv,\uv)_{\Hv(\opcurl)}.  \label{eq:defnormdisp}
\end{equation}

It is well known, that the space $\Hv(\opcurl)$ allows for the definition of a tangential trace. 
According to \cite{BuffaCiarlet:01b}, we may define the subspace of $\Hv(\opcurl)$ satisfying homogeneous tangential boundary conditions on $\Gamma_D$, which is our TDNNS displacement space
\begin{equation}
	\spaceV := \Hv_{0,\GammaD}(\opcurl) = \left\{ \vv \in \Lv^2(\Omega): \opcurl \vv \in \Lv^2(\Omega), \vv_\tang = 0\ \mbox{on}\ \GammaD \right\}.
\end{equation}
The following theorem is taken, with notation adapted to our work, from \cite[Theorem~6.6, Remark~6.7]{BuffaCiarlet:01b}:

\begin{theorem} \label{theo:tracehcurl}
The tangential trace operator $\vv \to \vv_\tang$ is bounded and surjective as a mapping
\begin{equation}
	\Hv(\opcurl) \to \Hv^{-1/2}_{\bot,00}(\opcurl,\Gamma_D) \label{eq:tracehcurlI}
\end{equation}
and
\begin{equation}
	\Hv_{0,\Gamma_D}(\opcurl) \to \Hv^{-1/2}_{\bot}(\opcurl^0,\Gamma_N). \label{eq:tracehcurlII}
\end{equation}
\end{theorem}

The first statement of Theorem~\ref{theo:tracehcurl}  ensures the existence of an extension $\tilde \uv_D$ of a given tangential-displacement boundary value. Additionally, Theorem~\ref{theo:tracehcurl} tells that the surface integral in \eqref{eq:spp2} can be understood as a duality product. We will elaborate on this matter in Section~\ref{sec:TDNNSAnalysis}.

A conforming finite element space for $\Hv(\opcurl)$ has to be tangential continuous, such as the N\'ed\'elec spaces introduced in \cite{Nedelec:80,Nedelec:86}.

An essential tool in the analysis of the TDNNS formulation is the \emph{regular decomposition}. Decompositions satisfying homogeneous Dirichlet or Neumann boundary conditions  have been shown by \cite{PasciakZhao:02} and \cite{Hiptmair:02}, respectively. The case of mixed boundary conditions can be found in \cite{HiptmairZheng:09}.

\begin{theorem}[regular decomposition] \label{theo:regulardec}
For $\uv \in \Hv_{0,\GammaD}(\opcurl)$ there exists a decomposition
\begin{equation}
	\vv = \zv + \nabla \phi,
\end{equation}
where $\zv \in \Hv^1_{0,\GammaD}(\Omega)$ and $\phi \in H^1_{0,\GammaD}(\Omega)$. The respective parts can be bounded by
\begin{equation}
	\|\phi\|_{H^1(\Omega)} \leq c \|\vv\|_{\Hv(\opcurl)} \qquad \mbox{and} \qquad \|\zv\|_{\Hv^1(\Omega)} \leq c \|\vv\|_{\Hv(\opcurl)}, \label{eq:regulardec}
\end{equation}
with a generic constant $c$. 
\end{theorem}

\subsection{The stress space $\Ht(\opdiv\opdivv)$}

We still need to specify the stress space. Roughly, it is a subspace of $\Lt^2$ where the (scalar-valued) divergence of the (vector-valued) divergence of the stress tensor lies in the dual space of $H^1_{0,\Gamma_D}$. In \eqref{eq:defnormstress}, the norm of the desired space is stated for smooth functions. We will proceed as follows: first, we formally define the space $\Ht(\opdiv\opdivv)$ as the closure of smooth functions, and give an interpretation of the norm which rectifies the name $\Ht(\opdiv\opdivv)$. Then, we show that the normal-normal trace can be bounded in this norm in the appropriate setting. Thus we can define the subspace $\Ht_{0,\Gamma_N}(\opdiv\opdivv)$ satisfying zero normal-normal boundary conditions on $\Gamma_N$ as the closure of smooth functions vanishing on $\Gamma_N$. Last, we provide an inverse trace theorem, which allows to extend normal-normal stress distributions from the boundary to the whole space $\Ht(\opdiv\opdivv)$.

The norm $\|\cdot\|_{\Ht(\opdiv\opdivv)}$ shall be defined for smooth $\taut \in C^{\infty}(\bar \Omega)$ by
\begin{equation}
	\|\taut\|_{\Ht(\opdiv\opdivv)}^2 = \|\taut\|_{\Lt^2}^2 + \left(\sup_{\varphi \in  H^2\cap H^1_{0,\GammaD}} \frac{\int_{\Omega} \taut:\epst(\nabla \varphi)\,d\xv - \int_{\partial \Omega} \tau_{nn} \frac{\partial \varphi}{\partial n}}{\|\nabla \varphi\|_{\Lv^2(\Omega)}}\right)^2 .
\label{eq:defnormstress}
\end{equation}

Note that, due to the symmetry of the Hessian $\nablat^2\varphi$, the symmetric expression $\epst(\nabla \varphi)$ is the same as the conventional notation $\nablat^2\varphi$. We use $\epst(\nabla \varphi)$, as it shows the relation to linear elasticity. 

We define the space $\Ht(\opdiv\opdivv)$ as
\begin{equation}
 \Ht(\opdiv\opdivv) := \overline{\underline{\mathbf{C}}^\infty_{sym}}^{\|\cdot\|_{\Ht(\opdiv\opdivv)}}.
\end{equation}

The second term in the definition of the norm \eqref{eq:defnormstress} is a seminorm and can be interpreted as the norm of $\opdiv\opdiv\taut$ in the dual space of $H^{1}_{0,\Gamma_D}$. Integration by parts of the denominator gives for smooth $\taut$, using that $\varphi$ and $\partial \varphi/\partial \tang$ vanish on $\Gamma_D$,
\begin{eqnarray}
\lefteqn{\int_\Omega \taut : \epst(\nabla \varphi)\, d\xv - \int_{\partial \Omega}\tau_{nn} \frac{\partial \varphi}{\partial n}\, ds}\\
&=& -\int_{\Omega} \opdivv \taut \cdot \nabla \varphi \, dx + \int_{\Gamma_N} \tauv_{\nv\tang} \cdot \frac{\partial \varphi}{\partial \tang}\, ds\\
&=& \int_{\Omega} \opdiv\opdivv \taut \,\varphi \, dx - \int_{\Gamma_N} (\opdivv\taut)_n\varphi\,ds + \int_{\Gamma_N} \tauv_{\nv\tang} \cdot \frac{\partial \varphi}{\partial \tang}\, ds
\end{eqnarray}

The supremum in \eqref{eq:defnormstress} can be interpreted as a dual norm: in the interior, $\opdiv\opdivv\taut$ is in the dual space of $H^1_{0,\Gamma_D}$. On $\Gamma_N$ we have $(\opdivv\taut)_n$ in the dual of the trace space $H^{1/2}_{00}(\Gamma_N)$. In the last term, the tangential derivative $\partial \varphi/\partial \tang$ appears. Since the gradient of $H^1$ lies in $\Hv(\opcurl)$, this tangential derivative is in $H^{-1/2}_{\bot}(\opcurl^0,\Gamma_N)$, see Theorem~\ref{theo:tracehcurl}. The normal-tangential stress $\tauv_{\nv\tang}$ is thus in the dual of this space, which means \cite{BuffaCiarlet:01b}
\begin{equation}
\tauv_{\nv\tang} \in [\Hv^{-1/2}_{\bot}(\opcurl^0,\Gamma_N)]^* = \Hv^{-1/2}_{\parallel,00}(\opdiv,\Gamma_N).
\end{equation}
We will comment on this restriction in Section~\ref{sec:fem}, as it our finite element space is not conforming in this term.

We will now define a space for the normal-normal trace, and show that the normal-normal trace is bounded in the $\Ht(\opdiv\opdivv)$ norm. To this end, we need the space of traces of the normal derivative of $H^2\cap H^1_{0,\Gamma_D}$,
\begin{equation}
	H^{1/2}_n(\partial\Omega) := \left\{ w = \frac{\partial \tilde w }{\partial n}: \tilde w \in H^2\cap H^1_{0,\Gamma_D}\right\}.
\end{equation}
For a polyhedral domain and $\Gamma_D = \emptyset$, $H^{1/2}_n(\partial\Omega)$ consists of piecewise $H^{1/2}(\Gamma_i)$ without continuity assumptions on the polyhedron edges or vertices, see e.g.\ \cite{GiraultRaviart:86}. For $\Gamma_D = \partial \Omega$, $H^{1/2}_n(\partial\Omega)$ is the subspace of the piecewise $H^{1/2}_{00}(\Gamma_i)$ spaces without continuity assumptions on the polyhedron edges or vertices, see e.g.\ \cite{Grisvard:85}.  To the best knowledge of the authors, this space has not been analyzed so far for general, nontrivial $\Gamma_D$. In this work, we only use that the space can be defined piecewise on each polyhedral face.

The normal-normal trace space of $\Ht(\opdiv\opdivv)$ is then given by
\begin{eqnarray}
	H^{-1/2}_n(\partial\Omega) &:=& [H^{1/2}_n(\partial \Omega)]^*.
\end{eqnarray}
An appropriate norm on $H^{-1/2}_n(\partial\Omega)$ is given by
\begin{equation}
 \|g\|_{H^{-1/2}_n(\partial\Omega)} = \sup_{\tilde w \in H^2\cap H^1_{0,\Gamma_D}} \frac{\langle g, \frac{\partial\tilde w}{\partial n}\rangle}{\|\nabla \tilde w\|_{H^1}}.
\end{equation}
Note that due to the piecewise nature of $H^{1/2}_n(\partial \Omega)$, the trace space can be restricted to each polyhedral face $\Gamma_i$, and extended from each face to the whole boundary by zero.

\begin{theorem} \label{theo:tracehdivdiv}
The normal-normal trace operator is bounded from $\Ht(\opdiv\opdivv)$ to $H^{-1/2}_n(\Gamma_i)$ for each boundary face $\Gamma_i \subset \partial \Omega$. Thus, it is well defined on $\Ht(\opdiv\opdivv)$ as the extension from $\Ct^{\infty}_{sym}$. For $\taut \in \Ht(\opdiv\opdivv)$ there holds the bound
\begin{equation}
	\|\tau_{nn}\|_{H^{-1/2}_n(\Gamma_i)} 
	\leq c \|\taut\|_{\Ht(\opdiv\opdivv)}
\end{equation}
with the constant $c$ independent of $\taut$.
\end{theorem}
\begin{proof}
Let $\taut \in \Ct^{\infty}_{sym}$ be fixed. We first show that the normal-normal trace on a boundary face $\Gamma_i$ can be bounded by the normal-normal trace on the whole boundary. Since $H^{1/2}_n(\partial\Omega)$ is a piecewise defined space without continuity assumptions between polyhedron faces, any $\varphi \in H^{1/2}_n(\Gamma_i)$ can be extended by zero to $\phi \in H^{1/2}_n(\partial \Omega)$. By definition of the dual norm we see
\begin{eqnarray}
	\|\tau_{nn}\|_{H^{-1/2}_n(\Gamma_i)} 
		&=& \sup_{\varphi\in H^{1/2}_n(\Gamma_i)} \frac{\int_{\Gamma_i} \tau_{nn}\varphi\,ds}{\|\varphi\|_{H^{1/2}_n(\Gamma_i)}}\\
		&=& \sup_{\varphi\in H^{1/2}_n(\partial\Omega)\atop\varphi=0\ \mathrm{on}\ \partial\Omega\backslash\Gamma_i} 
			\frac{\int_{\partial\Omega} \tau_{nn}\varphi\,ds}{\|\varphi\|_{H^{1/2}_n(\partial\Omega)}}\\
		&\leq& \sup_{\varphi\in H^{1/2}_n(\partial\Omega)} 
			\frac{\int_{\partial\Omega} \tau_{nn}\varphi\,ds}{\|\varphi\|_{H^{1/2}_n(\partial\Omega)}} = \|\tau_{nn}\|_{H^{-1/2}_n(\partial\Omega)}.
\end{eqnarray}
We proceed showing the actual trace inequality, where we use that $H^{1/2}_n(\partial\Omega)$ is defined as the trace space of $H^2\cap H^1_{0,\Gamma_D}$,
\begin{eqnarray}
\lefteqn{\|\tau_{nn}\|_{H^{-1/2}_n(\partial \Omega)} }\\
&=& \sup_{\varphi \in H^2\cap H^1_{0,\Gamma_D}} \frac{\int_{\partial \Omega} \tau_{nn} \frac{\partial \varphi}{\partial n}\,d\xv}{\|\nabla \varphi\|_{\Hv^1}}\\
&\leq& \sup_{\varphi \in H^2\cap H^1_{0,\Gamma_D}} \frac{-\int_\Omega \taut :\epst(\nabla \varphi) d\xv 
                                    +\int_{\partial\Omega} \tau_{nn} \frac{\partial \varphi}{\partial n}\,d\xv
																		}
																		{\|\nabla \varphi\|_{\Lv^2}} + \label{eq:38}\\
								&&				\sup_{\varphi \in H^2\cap H^1_{0,\Gamma_D}} \frac{
												            \int_\Omega \taut : \epst(\nabla \varphi)\,d\xv 
																		}
																		{\|\nabla \varphi\|_{\Hv^1}} \label{eq:39a}
\end{eqnarray}
We see that the supremum from eq. \eqref{eq:38} is already contained in the $\Ht(\opdiv\opdivv)$-norm. For the supremum from eq. \eqref{eq:39a}, we use that
\begin{equation}
\sup_{\varphi \in H^2\cap H^1_{0,\Gamma_D}} \frac{
												            \int_\Omega \taut : \epst(\nabla \varphi)\,d\xv 
																		}
																		{\|\nabla \varphi\|_{\Hv^1}}
																		\leq 
\sup_{\varphi \in H^2\cap H^1_{0,\Gamma_D}} \frac{
												            \int_\Omega \taut : \epst(\nabla \varphi)\,d\xv 
																		}
																		{\|\epst(\nabla \varphi)\|_{\Lt^2}}
																		\leq \|\taut\|_{\Lt^2}.
\end{equation}
Therefore, we arrive at the desired result
\begin{eqnarray}
\lefteqn{\|\tau_{nn}\|_{H^{-1/2}_n(\partial \Omega)} }\\
&\leq& \sup_{\varphi \in H^2\cap H^1_{0,\Gamma_D}} \frac{-\int_\Omega \taut :\epst(\nabla \varphi) d\xv 
                                    +\int_{\partial\Omega} \tau_{nn} \frac{\partial \varphi}{\partial n}\,d\xv
																		}
																		{\|\nabla \varphi\|_{\Lv^2}} +
												\|\taut\|_{\Lt^2}\\
&\leq& \sqrt{2}\|\taut\|_{\Ht(\opdiv\opdivv)}.
\end{eqnarray}

\end{proof}

The trace theorem above allows to define the space
\begin{equation}
	\Ht_{0,\Gamma_N}(\opdiv\opdivv) := \overline{\Ct^\infty_{sym,0,\Gamma_N}}^{\|\cdot\|_{\Ht(\opdiv\opdivv)}}.
\end{equation}
Any $\taut \in \Ht_{0,\Gamma_N}(\opdiv\opdivv)$ has a well-defined normal-normal trace in $H^{-1/2}_n(\partial\Omega)$, and it holds that $\tau_{nn} = 0$ in $H^{-1/2}_n(\Gamma_N)$.

Finally, we provide an inverse trace theorem for the space $\Ht(\opdiv\opdivv)$, before we proceed to the analysis of the  TDNNS elasticity problem. The inverse trace theorem allows to find an extension of a given (scalar) normal-normal stress on the boundary to a (tensor-valued) stress field on the domain.

\begin{theorem}
For any boundary face $\Gamma_i$, and $g \in H^{-1/2}_n(\Gamma_i)$, there exists a tensor field $\taut \in \Ht(\opdiv\opdivv)$ with $g = \tau_{nn}$ in the sense of $H^{-1/2}_n(\Gamma_i)$ and 
\begin{equation}
 \|\taut\|_{\Ht(\opdiv\opdivv)} \leq c \|g\|_{H^{-1/2}_n(\Gamma_i)}, \label{eq:invtrace}
\end{equation}
with constant $c$ independent of $g$.
\end{theorem}

\begin{proof}
For boundary face $\Gamma_i$, let $g \in H^{-1/2}_n(\Gamma_i)$ be given. The extension of $g$ by zero lies in $H^{-1/2}_n(\partial\Omega)$, which is the dual of the trace space of $H^2 \cap H^1_{0,\Gamma_D}$. This allows to pose the following problem in $H^2 \cap H^1_{0,\Gamma_D}$ with well-defined right hand side: find $w \in H^2 \cap H^1_{0,\Gamma_D}$ such that
\begin{align}
\int \epst(\nabla w): \epst(\nabla v)\, d\xv + \int_\Omega \nabla w \nabla v\, dx = \langle g, \frac{\partial v}{\partial n}\rangle && v \in H^2 \cap H^1_{0,\Gamma_D}. \label{eq:varH2}
\end{align}
Solvability of \eqref{eq:varH2} is clear, as we note that $\epst(\nabla w) = \nabla^2w$ due to the symmetry of the Hessian.
By the standard theory of Lax and Milgram we obtain the stability estimate
\begin{equation}
\|\epst(\nabla w)\|_{\Lt^2}^2 + \|\nabla w\|_{\Lv^2}^2 \leq c \|g\|_{H^{-1/2}_n(\Gamma_i)}^2.
\end{equation}
We choose $\taut = \epst(\nabla w)$, which is clearly symmetric and bounded in $\Lt^2$:
\begin{equation}
\|\taut\|_{\Lt^2}^2 \leq c \|g\|_{H^{-1/2}_n(\Gamma_i)}^2. \label{eq:39b}
\end{equation}
Additionally, it satisfies the natural boundary condition $\tau_{nn} = g$ in $H^{-1/2}_n(\Gamma_i)$. 
It remains to show that our choice of $\taut$ lies actually in $\Ht(\opdiv\opdivv)$, and satisfies the estimate \eqref{eq:invtrace}.
To this end, we still need to bound the supremum term in the definition of the norm. 
Any $\varphi \in C^{\infty}\cap H^1_{0,\Gamma_D}$ is a valid test function for the variational equation \eqref{eq:varH2}. This implies
\begin{eqnarray}
\lefteqn{ \sup_{\varphi \in C^{\infty}\cap H^1_{0,\Gamma_D}} \frac{\int_\Omega \taut : \epst(\nabla \varphi) d\xv - \langle \tau_{nn}, \frac{\partial \varphi}{\partial n} \rangle}{\|\nabla \varphi\|_{\Lv^2}} } \label{eq:40}\\
&=& \sup_{\varphi \in C^{\infty}\cap H^1_{0,\Gamma_D}} \frac{\int_\Omega \epst(\nabla w) : \epst(\nabla \varphi) d\xv - \langle g,  \frac{\partial \varphi}{\partial n} \rangle}{\|\nabla \varphi\|_{\Lv^2}} \\
&=& \sup_{\varphi \in C^{\infty}\cap H^1_{0,\Gamma_D}} \frac{ \int_\Omega \nabla w \cdot \nabla \varphi\, d\xv }{\|\nabla \varphi\|_{\Lt^2}} \\
&\leq&  \|\nabla w\|_{\Lv^2} \leq c \|g\|_{H^{-1/2}_n(\Gamma_i)}. \label{eq:43}
\end{eqnarray}
Adding up \eqref{eq:39b} and \eqref{eq:40}-\eqref{eq:43} leads to the desired bound \eqref{eq:invtrace}.
\end{proof}

With tools concerning $\Ht(\opdiv\opdivv)$ now at hand, we can proceed to the analysis of the variational problem \eqref{eq:spp1} -- \eqref{eq:spp2}.

\section{Analysis of the TDNNS problem} \label{sec:TDNNSAnalysis}

In the current section, we show existence and uniqueness of a solution to the TDNNS elasticity problem \eqref{eq:spp1} -- \eqref{eq:spp2}. We specify the variational spaces $\spaceSigma$ and $\spaceV$ foreshadowed in Section~\ref{sec:variational},
\begin{eqnarray}
	\spaceSigma &:=& \Ht_{0,\Gamma_N}(\opdiv\opdivv),\\
	\spaceV &:=& \Hv_{0,\Gamma_D}(\opcurl).
\end{eqnarray}
We shall use the theory on mixed problem treated in detail in \cite{BoffiBrezziFortin:13}. First, we concentrate on boundary conditions, then we show stability estimates for the bilinear forms $\bilA(\cdot,\cdot)$ and $\bilB(\cdot,\cdot)$. These results allow us to derive existence, uniqueness and stability of a solution to the TDNNS elasticity problem.

\subsection{Boundary conditions}

We assumed boundary conditions $\uv_D$ on $\Gamma_D$ and $\tv_{N}$ on $\Gamma_N$ to be given. We shall comment on the regularity necessary for these boundary conditions, such that the variational problem is well-defined. We treat the essential boundary conditions on tangential displacement and normal-normal stress first, and proceed to the natural boundary conditions on the normal displacement and normal-tangential stress afterwards.

In the variational formulation \eqref{eq:spp1} -- \eqref{eq:spp2}, we used extensions $\tilde \uv_D$ and $\tilde \sigmat_N$ of the given boundary data. The trace theorems for $\Hv(\opcurl)$ and $\Ht(\opdiv\opdivv)$ ensure that, given $\uv_{D,\tang} \in \Hv^{-1/2}_{\bot,00}(\opcurl,\Gamma_D)$ and $t_{N,n} \in H^{-1/2}_n(\Gamma_N)$, extensions can be found satisfying
\begin{align}
\tilde \uv_{D,\tang} &= \uv_{D,\tang}, &  \|\tilde \uv_{D}\|_{\Hv(\opcurl)} &\leq c \|\uv_{D,\tang}\|_{\Hv^{-1/2}_{\bot,00}(\opcurl,\Gamma_D)},\\
\tilde \sigma_{N,nn} &= t_{N,n}, & \|\tilde \sigmat_N\|_{\Ht(\opdiv\opdivv)} &\leq c \|t_{N,n}\|_{H^{-1/2}_n(\Gamma_N)}.
\end{align}

Natural boundary conditions on normal displacement and tangential component of normal stress are included into the right hand side of the variational problem \eqref{eq:spp1} -- \eqref{eq:spp2}. For smooth functions, they are included as surface integrals
\begin{equation}
	\int_{\Gamma_D} u_{D,n}\, \tau_{nn}\, ds\quad  \mbox{and}\quad \int_{\Gamma_N} \tv_{N,\tang}\cdot \vv_\tang\,ds
\end{equation}
We will see that both boundary integrals can be understood in the sense of duality products in the respective trace spaces, which makes them well-defined on the whole variational spaces. 

The trace theorem on $\Ht(\opdiv\opdivv)$, Theorem~\ref{theo:tracehdivdiv}, ensures that
\begin{equation}
 \langle u_{D,n}, \tau_{nn}\rangle_{H^{1/2}_n(\Gamma_D)\times H^{-1/2}_n(\Gamma_D)} \leq c \|u_n\|_{H^{1/2}_n(\Gamma_D)}\|\taut\|_{\Ht(\opdiv\opdivv)}.
\end{equation}
Thus, it is necessary to have the normal displacement $u_{D,n} \in H^{1/2}_n(\Gamma_D)$.

The second statement from Theorem~\ref{theo:tracehcurl}, \eqref{eq:tracehcurlII} ensures that for $\vv \in \spaceV$ the tangential trace $\vv_\tang$ allows for a surface curl, $\vv_\tang \in \Hv^{-1/2}_{\bot}(\opcurl^0,\Gamma_N)$. Therefore, the normal-tangential trace of the given surface tractions has to lie in its dual space, which is by \cite{BuffaCiarlet:01b} $[\Hv^{-1/2}_{\bot}(\opcurl^0,\Gamma_N)]^* = \Hv^{-1/2}_{\parallel,00}(\opdiv,\Gamma_N)$.
The trace theorem in $\Hv_{0,\Gamma_D}(\opcurl)$ ensures the bound
\begin{equation}
\langle \tv_{N,\tang}, \vv_\tang\rangle_{ \Hv^{-1/2}_{\parallel,00}(\opdiv,\Gamma_N)\times \Hv^{-1/2}_{\bot}(\opcurl^0,\Gamma_N)}
\leq c \|\tv_{N,\tang}\|_{ \Hv^{-1/2}_{\parallel,00}(\opdiv,\Gamma_N)} \|\vv\|_{\Hv(\opcurl)}.
\end{equation}
Let us shortly comment on the condition $\tauv_{\nv\tang} = \tv_{N,\tang} \in \Hv^{-1/2}_{\parallel,00}(\opdiv,\Gamma_N)$.
 This means that the given normal-tangential (shear) stress has to allow for a distributional surface divergence of some kind, which includes a continuity assumption on the in-plane normal across boundary edges. As the normal-tangential (shear) component of the proposed stress finite elements does not satisfy this condition, the finite element method is nonconforming, see Section~\ref{sec:fem}. Similarly, also the given shear stress $\tv_{N,\tang}$ does not need to satisfy any continuity assumptions in the finite element setting.

\subsection{Stability of bilinear forms}

To apply the theory on mixed systems by \cite{BoffiBrezziFortin:13}, we have to show 
\begin{itemize}
	\item boundedness of $\bilA(\cdot,\cdot)$ and $\bilB(\cdot,\cdot)$,
	\item coercivity of $\bilA(\cdot,\cdot)$ on the kernel space $\opker(\opB)$, and
	\item inf-sup stability of $\bilB(\cdot,\cdot)$.
\end{itemize}
There, the kernel space $\opker(\opB)$ is defined by
\begin{equation}
	\opker(\opB) := \{ \taut \in \spaceSigma: \bilB(\taut, \vv) = 0\quad \forall \vv \in \spaceV\}. \label{eq:kerB}
\end{equation}
The conditions on $\bilA(\cdot,\cdot)$ will follow rather quickly, assuming the elasticity matrix $\Atel$ to be regular, which is true for compressible materials with $\nu < \nu_0 < 1/2$. In the case of nearly incompressible materials with $\nu \to 1/2$, a refined analysis comparable to \cite[Chapter~5]{Sinwel:09} has to be carried out, which is not done in the scope of the present work. The estimates on $\bilB(\cdot,\cdot)$ are more involved.

\begin{lemma} \label{lemma:a}
The bilinear form $\bilA(\cdot,\cdot)$ is bounded on $\spaceSigma = \Ht_{0,\Gamma_N}(\opdiv\opdivv)$. Moreover, it is coercive on the kernel space $\opker(\opB)$ from \eqref{eq:kerB}, there exists a constant $c_a$ independent of $\taut$ such that
\begin{equation}
	\bilA(\taut,\taut) \geq c_a \|\taut\|_{\Ht(\opdiv\opdivv)}^2 \qquad \forall \taut \in \opker(\opB).
\end{equation}
\end{lemma}
\begin{proof}
Boundedness of $\bilA(\cdot,\cdot)$ in $\Ht(\opdiv\opdivv)$ is clear, since $\bilA(\cdot,\cdot)$ is a ($\Atel$-scaled) inner product on $\Lt^2$ and $\Ht(\opdiv\opdivv)$ is a subspace of $\Lt^2$. Obviously, $\bilA(\cdot,\cdot)$ is also coercive with respect to the $\Lt^2$ norm,
\begin{equation}
	\bilA(\taut,\taut) \geq c_a \|\taut\|_{\Lt^2}^2 \qquad \forall \taut \in \opker(\opB).
\end{equation}
To get coercivity with respect to the $\Ht(\opdiv\opdivv)$ norm, we need to show that the additional supremum term in \eqref{eq:defnormstress} vanishes for $\taut \in \opker(\opB)$. Since for $\varphi \in H^1_{0,\Gamma_D}$ we have $\nabla \varphi \in \spaceV = \Hv_{0,\Gamma_D}(\opcurl)$, we directly see from the definition of the kernel space \eqref{eq:kerB},
\begin{equation}
\sup_{\varphi \in H^2\cap H^1_{0,\Gamma_D}} \frac{\int_\Omega \taut : \epst(\nabla \varphi) d\xv - \langle \tau_{nn}, \frac{\partial \varphi}{\partial n} \rangle}{\|\nabla \varphi\|_{\Lv^2}} 
= \sup_{\varphi \in H^2\cap H^1_{0,\Gamma_D}} \frac{\bilB(\taut, \nabla \varphi)}{\|\nabla \varphi\|_{\Lv^2}}  = 0.
\end{equation}
This concludes the proof.
\end{proof}

Next, we treat boundedness of the bilinear form $\bilB(\cdot,\cdot)$.

\begin{lemma}
The bilinear form $\bilB: \Ht_{0,\Gamma_N}(\opdiv\opdivv) \times \Hv_{0,\Gamma_D}(\opcurl)$ is bounded.
\end{lemma}
\begin{proof}
Let $\taut \in \spaceSigma = \Ht_{0,\Gamma_N}(\opdiv\opdivv)$ and $\vv \in \spaceV = \Hv_{0,\Gamma_D}(\opcurl)$ be smooth, such that the integrals in \eqref{eq:defineBcontI} are well defined. We use the regular decomposition \eqref{eq:regulardec} $\vv = \zv + \nabla p$ with $\zv \in \Hv^1_{0,\Gamma_D}$ and $p \in H^1_{0,\Gamma_D}$, then
\begin{eqnarray}
	\bilB(\taut, \vv) &=& -\int_{\Omega} \taut : \epst(\zv)\, d\xv + \int_{\Gamma_D} \tau_{nn} \underbrace{z_n}_{=0}\, ds \\
		& & \int_{\Omega} \opdiv\taut \cdot \nabla p\, d\xv - \int_{\Gamma_N} \tau_{n\tang} \frac{\partial p}{\partial \tang}\, ds .
\end{eqnarray}
Cauchy's inequality in $\Lt^2$ and density ensure
\begin{eqnarray}
	\bilB(\taut, \vv) &\leq& \|\taut\|_{\Lt^2} \|\epst(\zv)\|_{\Lt^2}\\
		& & + \sup_{\varphi \in H^2 \cap H^1_{0,\Gamma_D}} \frac{ \int_{\Omega} \opdiv\taut \cdot \nabla \varphi\, d\xv - \int_{\Gamma_N} \tau_{n\tang} \frac{\partial \varphi}{\partial \tang}\, ds}{\|\nabla \varphi\|_{\Lt^2}} \|\nabla p\|_{\Lt^2} \\
		&=& \|\taut\|_{\Lt^2} \|\epst(\zv)\|_{\Lt^2}\\
		& & + \sup_{\varphi \in H^2 \cap H^1_{0,\Gamma_D}} \frac{ \int_{\Omega} \taut : \epst(\nabla \varphi)\, d\xv - \int_{\Gamma_D} \tau_{nn} \frac{\partial \varphi}{\partial n}\, ds}{\|\nabla \varphi\|_{\Lt^2}} \|\nabla p\|_{\Lt^2} .
\end{eqnarray}
We use the bound $\|\zv\|_{\Hv^1} + \|\nabla p\|_{\Lt^2} \leq c \|\vv\|_{\Hv(\opcurl)}$, and arrive at
\begin{equation}
\bilB(\taut, \vv) \leq c \|\taut\|_{\Ht(\opdiv\opdivv)} \|\vv\|_{\Hv(\opcurl)}.
\end{equation}
\end{proof}

This continuity result allows us to extend the bilinear form from smooth functions to the whole of $\Ht_{0,\Gamma_N}(\opdiv\opdivv) \times \Hv_{0,\Gamma_D}(\opcurl)$ in the sense of a distributional divergence operator.

\begin{lemma} \label{lem:binfsup}
The bilinear form $\bilB: \Ht_{0,\Gamma_N}(\opdiv\opdivv) \times \Hv_{0,\Gamma_D}(\opcurl)$ is inf-sup stable, for any $\vv \in \spaceV = \Hv_{0,\Gamma_D}(\opcurl)$ there exists some $\sigmat \in \spaceSigma= \Ht_{0,\Gamma_N}(\opdiv\opdivv)$ such that
\begin{equation}
	\bilB(\sigmat, \vv) \geq c_b \|\sigmat\|_{\Ht(\opdiv\opdivv)} \|\vv\|_{\Hv(\opcurl)}.
\end{equation}
The constant $c_b > 0$ is independent of $\vv$.
\end{lemma}
\begin{proof}
Let $\vv \in \Hv_{0,\Gamma_D}(\opcurl)$ be fixed. Find $\uv \in \Hv^1 \cap \Hv_{0,\Gamma_D}(\opcurl)$ as a solution to the primal elasticity problem that for all $\tilde\vv \in \Hv^1 \cap \Hv_{0,\Gamma_D}(\opcurl)$
\begin{equation}
\int_{\Omega} \epst(\uv):\epst(\tilde \vv)\,dx + \int_{\Gamma_D} u_n \tilde v_n \,d\xv=  \int_{\Omega} \opcurl \vv \cdot \opcurl \tilde \vv + \vv \cdot \tilde \vv\,dx. \label{eq:varel}
\end{equation}
This solution satisfies the following ``classical'' boundary conditions: On $\Gamma_N$, we have a free boundary, with $(\epst(\uv))_\nv = 0$, on $\Gamma_D$, the tangential displacement $\uv_\tang = 0$ is fixed, while the normal displacement satisfies $u_n = -(\epst(\uv))_{nn}$. In the variational setting, we have the combined boundary condition
\begin{equation}
 \int_{\Gamma_D} u_n \tilde v_n\,ds = -\int_{\partial \Omega} (\Ctel \epst(\uv))_{nn} \tilde v_n\,ds  \qquad \forall \tilde \vv \in \Hv^1 \cap \Hv_{0,\Gamma_D}(\opcurl), \label{eq:sigmann}
\end{equation}
where the surface integral is to be understood as a duality product.

We choose $\sigmat := \Ctel \epst(\uv)$. We have to show that $\sigmat$ lies in $\Ht_{0,\Gamma_N}(\opdiv\opdivv)$. To this end, we first prove that it satisfies the $\Ht(\opdiv\opdivv)$-essential boundary condition $\sigma_{nn} = 0$ in $H^{-1/2}_{n}(\Gamma_N)$, then we proceed to bound $\sigmat$ in the $\Ht(\opdiv\opdivv)$ norm.

The tensor field $\sigmat$ satisfies $\sigma_{nn} = 0$ in $H^{-1/2}_{n}(\Gamma_N)$ if and only if for each $\varphi \in H^{1/2}_n(\Gamma_N)$ there holds 
\begin{equation}
	\langle \sigma_{nn}, \varphi\rangle_{H^{-1/2}_{n}(\Gamma_N)\times H^{1/2}_{n}(\Gamma_N)} = 0. \label{eq:sigmann0}
\end{equation}
Due to our assumptions on $\Omega$ and $\Gamma_N$, $\varphi$ can be extended to $H^{1/2}_n(\partial\Omega)$ by zero. By definition of $H^{1/2}_n(\partial\Omega)$, there exists a $\tilde \varphi \in H^2\cap H^1_{0,\Gamma_D}$ with $\varphi = \partial\tilde\varphi/\partial n$.  Since $\nabla H^1_{0,\Gamma_D} \subset \Hv_{0,\Gamma_D}(\opcurl)$, its gradient $\nabla \tilde \varphi$ is a valid test function in \eqref{eq:sigmann}, which ensures \eqref{eq:sigmann0}.

We now want to bound $\|\sigmat\|_{\Ht(\opdiv\opdivv)}$. Standard theory for the primal problem \eqref{eq:varel} ensures
\begin{equation}
\|\sigmat\|_{\Lt^2}^2 \leq c \|\vv\|_{\Hv(\opcurl)}^2 . \label{eq:LM}
\end{equation}
We bound the supremum term in the $\Ht(\opdiv\opdivv)$ norm to show that $\sigmat \in \Ht(\opdiv\opdivv)$: 
Again, $\nabla \varphi$ is a valid test function for the variational problem \eqref{eq:varel},
\begin{eqnarray}
	\lefteqn{\sup_{\varphi \in H^2 \cap H^1_{0,\Gamma_D}} \frac{\int_\Omega \sigmat : \epst(\nabla \varphi) d\xv 
			- \int_{\partial \Omega} \sigma_{nn} \frac{\partial \varphi}{\partial \nv}\,ds}{\|\nabla \varphi\|_{\Lv^2}}}\\
&=&\sup_{\varphi \in H^2 \cap H^1_{0,\Gamma_D}} \frac{\int_\Omega \Ctel\epst(\uv) : \epst(\nabla \varphi) d\xv 
			- \int_{\Gamma_D} (\Ctel\epst(\uv))_{nn} \frac{\partial \varphi}{\partial \nv}\,ds}{\|\nabla \varphi\|_{\Lv^2}} \\
&=&\sup_{\varphi \in H^2 \cap H^1_{0,\Gamma_D}} \frac{\int_\Omega \opcurl \vv \cdot \overbrace{\opcurl \nabla \varphi}^{=0} +  \vv \cdot \nabla \varphi d\xv 
			}{\|\nabla \varphi\|_{\Lv^2}}\\
&\leq& \|\vv\|_{\Lv^2} .
\end{eqnarray}
Together with \eqref{eq:LM} this leads to the bound
\begin{equation}
	\|\sigmat\|_{\spaceSigma} \leq c \|\vv\|_{\Hv(\opcurl)}.
\end{equation}
The bilinear form $\bilB(\sigmat,\vv)$ evaluates to
\begin{eqnarray}
\bilB(\sigmat, \vv) 
& = & \int_\Omega \sigmat:\epst(\vv)\,dx - \int_{\Gamma_D} \sigma_{nn} v_n\,ds\\
& \geq &  \|\opcurl\vv\|_{\Lv^2}^2 + \|\vv\|_{\Lv^2}^2\\
& \geq &  \|\sigmat\|_{\Ht(\opdiv\opdivv)} \|\vv\|_{\Hv(\opcurl)}.
\end{eqnarray}
\end{proof}

\section{Finite element spaces and their norms} \label{sec:fem}

Let $\triang = \{T\}$ be a simplicial, (shape-)regular triangulation of $\Omega$ as defined in  \cite[Def.~5.11]{Monk:03}. We denote the set of element faces $\faces = \{F\}$. 
Any piecewise smooth vector field $\vv \in \Hv(\opcurl)$ has to be tangential continuous on element interfaces, i.e.\ $\vv_\tang$ is uniquely defined on each element interface.
In \cite{Sinwel:09,PechsteinSchoeberl:11}, it was shown that a piecewise smooth tensor $\taut \in \Ht(\opdiv \opdivv)$ is normal-normal continuous across element faces, i.e. the normal-normal component $\tau_{nn}$ is continuous.

So far, the bilinear form $\bilB(\cdot,\cdot)$ is defined only for smooth vector and tensor fields in \eqref{eq:defineBcontI}, \eqref{eq:defineBcontII}. This definition can be extended to piecewise smooth fields:
Let $\vv \in \Hv_{0,\GammaD}(\opcurl)$ and $\taut \in \Ht_{0,\GammaN}(\opdiv\opdivv)$ be piecewise smooth and tangential and normal-normal continuous on the triangularization $\triang$, respectively, then
\begin{align}
	\bilB(\taut, \vv) &= \sum_{T \in \triang} \left( \int_T \opdivv \taut \cdot \vv\, d\xv - \int_{\partial T} \tauv_{\nv\tang}\cdot \vv_\tang\, ds \right) \label{eq:evaluatediv_1}\\
	&= -\sum_{T \in \triang} \left( \int_T \taut : \epst(\vv)\, d\xv - \int_{\partial T} \tau_{nn} v_n\, ds \right). \label{eq:evaluatediv_2}
\end{align}

In \cite{Sinwel:09,PechsteinSchoeberl:11}, normal-normal continuous symmetric stress finite elements were constructed, which  were used together with N\'ed\'elec elements for the displacement. A stable finite element method was obtained. However, the method is slightly nonconforming: to be in $\Ht(\opdiv\opdivv)$, a piecewise continuous function has to be normal-normal continuous, and the normal-tangential component $\tauv_{\nv\tang}$ has to lie in the dual space of the trace space of $\Hv(\opcurl,T)$, which means  $\tauv_{\nv\tang} \in \Hv^{-1/2}_{||}(\opdiv_{\partial T})$. Then the surface integral in \eqref{eq:evaluatediv_1} can be understood as duality product and evaluated for arbitrary $\vv \in \Hv_{0,\GammaD}(\opcurl)$. However, this is a continuity restriction on $\tauv_{\nv\tang}$ at element edges, which does not hold for general $\taut$ piecewise smooth. In general, the surface vector field $\tauv_{\nv\tang}$ is discontinuous across element edges.

We choose the following finite element spaces for integer $k \geq 1$
\begin{align}
	\spaceSigma_h &:= \left\{ \taut_h \in \Lt^2(\Omega): \taut_h|_T \in \Pt^k(T), \tau_{h,nn} \mbox{ cont.}, \tau_{h,nn}|_{\GammaN} = 0 \right\};\\
	\spaceV_h &:= \left\{ \vv_h \in \Lv^2(\Omega): \vv_h|_T \in \Pv^k(T), \vv_{h,\tang} \mbox{ cont.}, \vv_{h,\tang}|_{\GammaD} = 0 \right\}.
\end{align}
Additionally, we need an $H^1$-conforming scalar finite element space $\spaceW_h$ of order $k+1$ satisfying zero boundary conditions,
\begin{equation}
	\spaceW_h := \left\{ w_h \in L^2(\Omega): w_h|_T \in P^{k+1}(T), w_h \mbox{ cont.}, w_h|_{\GammaD} = 0\right\}.
\end{equation}	
For this choice we have that $\nabla \spaceW_h \subset \spaceV_h$.

Utilizing the finite element spaces above, we may pose the finite element problem (restricting ourselves to the case of trivial essential boundary conditions $\uv_\tang = 0$ on $\Gamma_D$ and $\sigma_{nn} = 0$ on $\Gamma_N$): find $\sigmat_h \in \spaceSigma_h$ and $\uv_h \in \spaceV_h$ such that
\begin{align}
	\bilA(\sigmat_h, \taut_h) + \bilB(\taut_h, \uv_h) &= \int_{\Gamma_D} \tau_{nn} u_{D,n}\,ds && \forall \taut_h \in \spaceSigma_h, \label{eq:spph1}\\
	\bilB(\sigmat_h, \vv_h) &= \int_\Omega \fv \cdot \vv_h\,d\xv + \int_{\Gamma_N} \tv_{N,\tang}\cdot\vv_\tang\,ds&& \forall \vv_h \in \spaceV_h. \label{eq:spph2}
\end{align}

\subsection{Discrete stress norm}

While the N\'ed\'elec space $\spaceV_h$ and the continuous space $\spaceW_h$ are endowed with  $\Hv(\opcurl)$ and $H^1$ norms, respectively, we provide a discrete norm for the stress space $\spaceSigma_h$,
\begin{align}
	\|\taut_h\|_{\spaceSigma_h}^2 &:= \|\taut_h\|_{\Lt^2}^2 + \sum_{F \in \faces} h_F \|\tau_{h,nn}\|_{L^2(F)}^2  + \left(\sup_{w_h \in \spaceW_h}\frac{\bilB(\taut_h, \nabla w_h)}{\|\nabla w_h\|_{\Lv^2}}\right)^2. \label{eq:normsigmah}
\end{align}
Note that, for finite element functions $\taut_h$, the face terms in \eqref{eq:normsigmah} can be bounded by the $\Lt^2$ term (see \cite{PechsteinSchoeberl:11}). Thus, the face terms may be omitted for finite element functions, which is often done in this work. However, this is not possible for general piecewise smooth normal-normal continuous tensor fields $\taut$.

In \cite{PechsteinSchoeberl:11}, we showed stability of the problem not using the $\Hv(\opcurl)$ and discrete $\Ht(\opdiv\opdivv)$ norm \eqref{eq:normsigmah}, but using the $\Lt^2$ norm for the stresses and a broken $\Hv^1$ norm for the displacements,
\begin{align}
	\|\vv_h\|_{\Hv^1,h}^2 &= \sum_{T \in \triang} \|\epst(\vv_h)\|_{\Lt^2(T)}^2 + h^{-1}\sum_{F \in \faces} \|[v_{h,n}]\|_{L^2(F)}^2. \label{eq:normh1h}
\end{align}

\subsection{The reference element and transformations to the mesh element}

We introduce the reference tetrahedron $\hat T = \{ \hat \xv = (\hat x_1, \hat x_2, \hat x_3): \hat x_i>0, \hat x_1 + \hat x_2 + \hat x_3 < 1\}.$ Barycentric coordinates on $\hat T$ are given by
\begin{equation}
\hat \lambda_1 = \hat x_1,\quad \hat \lambda_2 = \hat x_2,\quad \hat \lambda_3 = \hat x_3,\quad \hat \lambda_4 = 1-(\hat x_1+\hat x_2 + \hat x_3). 
\end{equation}
For any element $T \in \triang$, let
\begin{equation}
	\Phi_T: \hat T \to T,\ \hat \xv \mapsto \xv
\end{equation}
be a smooth one-to-one mapping of the reference tetrahedron to tetrahedron $T$. The Jacobian of this transformation shall be denoted by $\Ft_T = \nabla \Phi_T$, the Jacobi determinant by $J_T = det(\Ft_T)$. The local mesh size is defined as the spectral norm of $\Ft_T$, $h_T = |\Ft_T|_s$. For a face $F$ and an edge $E$, let $J_F$, $J_E$ denote the transformation of measures of the mappings $\hat F \to F$, $\hat E \to E$. For the normal $\nv_F$ to face $F$ and the tangential vector $\tang_E$ to some edge $E$ we have
\begin{equation}
	\nv_F = J_T/J_F \Ft_T^{-T} \hat \nv_{\hat F}, \qquad \tang_E = 1/J_E \Ft_T \hat \tang_{\hat E}.
\end{equation}

The finite element basis functions are defined on the reference tetrahedron, and mapped to an element $T$ by a conforming transformation. A conforming transformation has to preserve the degrees of freedom of the finite element. While $H^1$ conforming elements can be transformed directly, we need the tangential-trace preserving covariant transformation for $\Hv(\opcurl)$ conforming elements, and a transformation which preserves the normal-normal trace for the stress elements,
\begin{align}
w_h(\xv) &= \hat w(\hat \xv),\\
\vv_h(\xv) &= \Ft_T^{-T} \hat \vv_h(\hat \xv),\\
\taut_h(\xv) &= 1/J_T^2 \Ft_T \hat \taut_h(\hat \xv) \Ft_T^T. \label{eq:transhdivdiv}
\end{align}
By application of basic calculus one can see that gradient and strain operators transform as
\begin{align}
\nabla w_h &= \Ft_T^{-T} \hat \nabla \hat w_h,\\
\epst(\vv_h) &= \Ft_T^{-T} \hat \epst(\hat \vv_h) \Ft_T^{-1}.
\end{align}

\section{Interpolation operators}

The a-priori error analysis of the proposed finite element method relies on interpolation operators for the finite element spaces. Subsequently, we recall the nodal interpolation operators for the spaces $\spaceW_h$ and $\spaceV_h$, as well as the Cl\'ement quasi-interpolation operator for the piecewise linear, continuous finite element space. Their definitions and according estimates can be found in \cite[Sect.~5.5 and 5.6]{Monk:03}. Additionally, we present an error estimate for the $\Hv(\opcurl)$-interpolant in the broken $\Hv^1$ norm. In Section~\ref{sec:interpol_stress} we define an interpolation operator for the stress space and give error estimates in the discrete stress norm $\|\cdot\|_{\spaceSigma_h}$. 

\subsection{Commuting interpolation operators for $H^1$ and $\Hv(\opcurl)$ and the Cl\'ement quasi-interpolation operator}


Let $\IntW$, $\IntV$ be the nodal interpolation operators defined on the finite element spaces  $\spaceW_h, \spaceV_h$.
These interpolation operators are based on the degrees of freedom of the finite element spaces. They are defined in such a way that they commute with the gradient operator,
\begin{equation}
\IntV \nabla w = \nabla \IntW w.
\end{equation}
Note that the interpolation operators are not well-defined for general functions in $H^1$ and $\Hv(\opcurl)$, respectively, but only for smoother functions allowing for point values or mean values of the tangential component along element edges. 
The Cl\'ement quasi-interpolation operator $\Clement$ is defined for general functions in $H^1$, as it uses mean values instead of nodal values. It is continuous in $H^1$.


The following interpolation error estimates are well known: 
\begin{theorem}
For $w \in H^{s+1}$ and $\vv \in \Hv^s, \opcurl \vv \in \Hv^s$ the following interpolation error estimates for $1 \leq s \leq k$ hold
\begin{align}
\|w - \IntW w\|_{H^1(\Omega)} & \leq c \left( \sum_{T \in \triang} h_T^{2s} \|w\|_{H^{s+1}(T)}^2 \right)^{1/2},\\
\|\vv - \IntV \vv\|_{\Hv(\opcurl)} & \leq c  \left(\sum_{T \in \triang} h_T^{2s} (\|\vv\|_{\Hv^{s}(T)}^2+\|\opcurl \vv\|_{\Hv^{s}(T)}^2) \right)^{1/2},\\
\|w - \Clement w\|_{L^2(\Omega)} & \leq c \left( \sum_{T \in \triang} h_T^{2} \|w\|_{H^{1}(D_T)}^2 \right)^{1/2}.
\end{align}
Here, $D_T$ denotes the neighbourhood of element $T$, i.e.\ the union of all elements sharing at least a vertex with $T$. The constants $c$ are independent of the mesh size $h$.
\end{theorem}

In \cite{Sinwel:09}, we showed that the N\'ed\'elec interpolator $\IntV$ also approximates in the broken $H^1$ norm,
\begin{theorem}
Let  $\vv \in \Hv^s, \opcurl \vv \in \Hv^s$ satisfy $\vv \in \Hv^{s+1}(T)$ for all elements $T \in \triang$. Then  the interpolation error is bounded in the broken $H^1$ norm \eqref{eq:normh1h} for $1 \leq s \leq k$
\begin{align}
\|\vv - \IntV \vv\|_{\Hv^1,h} & \leq c \left( \sum_{T \in \triang} h_T^{2s} \|\epst(\vv)\|_{\Ht^s(T)}^2 \right)^{1/2} .
\end{align}
\end{theorem}

\subsection{An interpolation operator for the stress space} \label{sec:interpol_stress}

We characterize the stress interpolation operator $\IntSigma$ and we show that it approximates not only in the $\Lt^2$ norm, but also in the discrete $\Ht(\opdiv\opdivv)$ norm $\|\cdot\|_{\spaceSigma_h}$ defined in \eqref{eq:normsigmah}.

We define six constant tensor fields on the reference tetrahedron, which are linearly independent and span the space of constant symmetric tensor fields. Four of these tensors are associated to a face of the tetrahedron, each. They are denoted by $\hat \St^{\hat F_m}, m=1\dots4$. The normal-normal component of a tensor field $\hat \St^{\hat F_m}$ is constant on face $\hat F_m$, while it vanishes on all other faces,
\begin{equation}
	\hat S^{\hat F_m}_{\hat n \hat n} |_{F_i} = c\delta_{i,m} \qquad \mbox{for}\ i,m=1\dots4.
\end{equation}
The remaining two tensor fields $\hat \St^{\hat T,n}, n=1,2$ have a vanishing normal-normal component on all element faces. 
\begin{equation}
	\hat S^{\hat T,n}_{\hat n \hat n} |_{F_i} = 0 \qquad \mbox{for}\ i=1\dots4, n=1,2.
\end{equation}
Moreover, the face tensors $\hat \St^{\hat F_m}$ are orthogonal to the interior tensors $\hat \St^{\hat T,n}$ in the sense that
\begin{equation}
	\hat \St^{\hat F_m} : \hat \St^{\hat T,n} = 0 \qquad \mbox{for}\ m=1\dots4, n=1,2. \label{eq:orthogonalS}
\end{equation}
The tensor fields are given by
\begin{align}
\hat \St^{\hat F_1} &= \left( \begin{array}{ccc} -6&1&1 \\ 1&0&1 \\ 1&1&0 \end{array}\right) ,&
\hat \St^{\hat F_2} &= \left( \begin{array}{ccc} 0&1&1 \\ 1&-6&1 \\ 1&1&0 \end{array}\right) ,&
\hat \St^{\hat F_3} &= \left( \begin{array}{ccc} 0&1&1 \\ 1&0&1 \\ 1&1&-6 \end{array}\right) ,\\
\hat \St^{\hat F_4} &= \left( \begin{array}{ccc} 0&1&1 \\ 1&0&1 \\ 1&1&0 \end{array}\right) ,&
\hat \St^{\hat T,1} &= \left( \begin{array}{ccc} 0&0&-1 \\ 0&0&1 \\ -1&1&0 \end{array}\right) ,&
\hat \St^{\hat T,2} &= \left( \begin{array}{ccc} 0&-1&0 \\ -1&0&1 \\ 0&1&0 \end{array}\right) .
\end{align}

The interpolation operator $\IntSigma$ mapping any sufficiently smooth, normal-normal-continuous tensor field $\taut$ to $\IntSigma \taut \in \spaceSigma_h$, is uniquely defined by the following conditions,
\begin{itemize}
	\item on each face $F \in \faces$, 
	\begin{equation}
		\int_F J_F (\taut - \IntSigma \taut)_{nn} q\, ds = 0 \qquad \mbox{for all}\ q \in P^k(F). \label{eq:facecond}
	\end{equation}
	\item on each element $T \in \triang$
	\begin{align}
		\int_T J_T (\taut - \IntSigma \taut) : (q \Ft_T^{-T} \hat \St^{\hat F_m} \Ft_T^{-1})\, d\xv &=0 &&\mbox{for all}\ q \in P^{k-1}(T), m=1\dots4, \label{eq:innercondI}\\
		\int_T J_T (\taut - \IntSigma \taut): (q \Ft_T^{-T} \hat \St^{\hat T,n} \Ft_T^{-1})\, d\xv &=0 &&\mbox{for all}\ q \in P^{k}(T), n=1,2. \label{eq:innercondII}
	\end{align}
\end{itemize}

\begin{lemma}
The interpolation operator $\IntSigma$ is well-defined and preserves piecewise polynomials, i.e. $\IntSigma \taut_h = \taut_h$ for $\taut_h \in \spaceSigma_h$.
\end{lemma}
\begin{proof}
We show that the conditions \eqref{eq:facecond}, \eqref{eq:innercondI}, \eqref{eq:innercondII} are unisolvent for the finite element space $\spaceSigma_h$. It is sufficient to show that \eqref{eq:facecond}, \eqref{eq:innercondI}, \eqref{eq:innercondII} applied to $\taut_h \in \spaceSigma_h$ implies $\taut_h = 0$.

We start with the face-bound conditions. On each face $F \in \faces$, we have
\begin{equation}
		\int_F J_F \tau_{h,nn} q\, ds = 0 \qquad \mbox{for all}\ q \in P^k(F). 
\end{equation}
Since $\tau_{h,nn}$ is polynomial of order $k$ on each face, this implies that $\tau_{h,nn} = 0$ on each face.

Since the normal-normal component of $ \taut_h$ vanishes on all element interfaces, on each element $T \in \triang$, $\taut_h$ is a linear combination of element-local interior shape functions. According to \cite{PechsteinSchoeberl:11}, there are two types of interior shape functions on the reference element,
\begin{align}
& \hat p_{i} \hat \lambda_m \hat \St^{\hat F_m},&& \hat p_{i}\ \mbox{basis for}\ P^{k-1}(\hat T), m=1\dots4  \label{eq:shapeI}\\
 & \hat p_{i} \hat \St^{\hat T,n},&& \hat p_{i}\ \mbox{basis for}\ P^{k}(\hat T), n=1,2. \label{eq:shapeII}
\end{align}
Transforming the integrals \eqref{eq:innercondI}, \eqref{eq:innercondII} to the reference element using the $\Ht(\opdiv\opdivv)$ conforming transformation \eqref{eq:transhdivdiv} leads to
\begin{align}
	\int_{\hat T} \hat \taut_h : ( \hat p_{i}  \hat \St^{\hat F_m})\, d\hat \xv &=0 &&\hat p_{i}\ \mbox{basis for}\  P^{k-1}(\hat T), m=1\dots4, \label{eq:innercondIhat}\\
	\int_{\hat T} \hat \taut_h : ( \hat p_{i}  \hat \St^{\hat T,n} )\, d\hat \xv &=0 &&\hat p_{i}\ \mbox{basis for}\  P^{k}(\hat T), n=1,2. \label{eq:innercondIIhat}
\end{align}
To show that $\taut_h = 0$, conditions \eqref{eq:innercondIhat}, \eqref{eq:innercondIIhat} are evaluated for all shape functions \eqref{eq:shapeI}, \eqref{eq:shapeII}, the results stored in a square but non-symmetric matrix. Since the tensor fields are orthogonal \eqref{eq:orthogonalS}, the two groups decouple, leaving two matrices of block structure,
\begin{equation}
\left[ \hat \St^{\hat F_m}:\hat \St^{\hat F_{\bar m}} \int_{\hat T} \hat \lambda_m \hat p_i \hat p_{\bar i} \,d\hat \xv\right]_{\substack{m,i\\ \bar m, \bar i}},
\qquad
\left[ \hat \St^{\hat T,n}:\hat \St^{\hat T,\bar n} \int_{\hat T} \hat p_i \hat p_{\bar i} \,d\hat \xv \right]_{\substack{n,i\\ \bar n, \bar i}}
\end{equation}
The regularity of these matrices can easily be shown using the linear independence of the tensor fields $\hat \St^{\hat F_m}, \hat \St^{\hat T,n}$, the positivity of the barycentric coordinates $\hat \lambda_m$, and the linear independence of the basis $\{\hat p_i\}$.
\end{proof}

\begin{theorem}
For $\taut \in \Hv^s$ and $1 \leq s \leq k+1$ the interpolation error is bounded by
\begin{align}
\|\taut - \IntSigma \taut\|_{\spaceSigma_h} & \leq c \left( \sum_{T \in \triang} h_T^{2s} \|\taut\|_{H^{s}(T)}^2 \right)^{1/2}.
\end{align}
\end{theorem}

\begin{proof}
In \cite{Sinwel:09}, it was shown that a very similar interpolation operator approximates in the $\Lt^2/L^2(F)$ norm. The same estimates holds for $\IntSigma$, which is expected, since the local space is the full polynomial space of order $k$ and $\IntSigma$ preserves piecewise polynomial finite element functions. The proof relies on a scaling argument and the Bramble-Hilbert lemma applied on the reference element, and is not provided in detail here,
\begin{equation}
\|\taut - \IntSigma \taut\|_{\Lt^2(\Omega)}^2 + \sum_{F \in \faces} h_F \|(\taut - \IntSigma \taut)_{nn}\|_{L^2(F)}^2   \leq c \ \sum_{T \in \triang} h_T^{2s} \|\taut\|_{H^{s}(T)}^2.
\end{equation}
To estimate the full norm $\|\taut - \IntSigma \taut\|_{\spaceSigma_h}$, we show that the supremum term vanishes,
\begin{equation}
\sup_{w_h \in \spaceW_h} \frac{\bilB(\taut - \IntSigma \taut, \nabla w_h)}{\|\nabla w_h\|_{\Lv^2(\Omega)}} \stackrel{!}{=} 0.
\end{equation}
We observe, due to the definition of the interpolation operator $\IntSigma$,
\begin{align}
&\bilB(\taut - \IntSigma \taut, \nabla w_h´) \\
&= \sum_{T \in \triang} \left( \int_T (\taut - \IntSigma \taut) : \underbrace{\epst(\nabla w_h)}_{\in P^{k-1}}\, d\xv - \int_{\partial T} (\taut - \IntSigma \taut)_{nn} \underbrace{\frac{\partial w_n}{\partial n}}_{\in P^k}\, ds \right)\\
& = 0
\end{align}
The stress interpolation operator $\IntSigma$ is defined in such a way that $\bilB (\taut - \IntSigma \taut, \nabla w_h)$ vanishes for any $w_h \in \spaceW_h$. Thus, the interpolation error estimate in the natural norm coincides with the estimate in $\Lt^2$ norm.
\end{proof}

\section{Analysis of the finite element problem}

A crucial tool for the analysis of the finite element problem is a discrete version of the regular decomposition from Theorem~\ref{theo:regulardec}. The following discrete decomposition can be deduced directly from the regular decomposition, see \cite{KolevVassilevski:09} for the case of $\Gamma = \Gamma_D$.

\begin{lemma} \label{lem:vas}
For a finite element vector field $\vv_h \in \spaceV_h$, there exists a decomposition
\begin{equation}
	\vv_h = \IntV \zv + \nabla p_h,
\end{equation}
with $\zv\in \Hv^1_{0,\GammaD}$, $\opcurl \zv = \opcurl \vv_h$ and $p_h \in W_h$. The respective parts can be bounded by
\begin{equation}
	\|p_h\|_{H^1} \leq c \|\vv_h\|_{\Lv^2} \qquad \mbox{and} \qquad 
	\|\zv\|_{\Hv^1} \leq c \|\vv_h\|_{\Hv(\opcurl)},
\end{equation}
with a generic constant $c$.
\end{lemma} 

\subsection{Continuity of the finite element problem}

We are concerned with continuity of the  bilinear forms with respect to the discrete norm $\|\cdot\|_{\spaceSigma_h}$ and the $\Hv(\opcurl)$ norm $\|\cdot\|_{\Hv(\opcurl)}$.
Obviously, $\bilA(\cdot,\cdot)$ is continuous, as it is continuous in $\Lt^2$.
For $\bilB(\cdot,\cdot)$, showing continuity is more challenging. 

\begin{lemma} \label{lemma:bcont}
The bilinear form $\bilB(\cdot,\cdot)$ defined in \eqref{eq:evaluatediv_1}, \eqref{eq:evaluatediv_2} is continuous on $\spaceSigma_h \times \spaceV_h$ with respect to the norms $\|\cdot\|_{\Hv(\opcurl)}$ and $\|\cdot\|_{\spaceSigma_h}$. For $\taut_h \in \spaceSigma_h$ and $\vv_h \in \spaceV_h$ there exists a constant $c$ independent of mesh size $h$
\begin{equation}
	\bilB(\taut_h, \vv_h) \leq c \|\taut_h\|_{\spaceSigma_h} \|\vv_h\|_{\Hv(\opcurl)}. \label{eq:bcont}
\end{equation}
Estimate \eqref{eq:bcont} can be generalized to any piecewise smooth, normal-normal continuous tensor field $\taut$, $\taut|_T \in \Ht^1_{sym}(T)$ for all $T\in \triang$, $\tau_{nn}|_F \in L^2(F)$.
\end{lemma}
\begin{proof}
Let $\taut$, $\taut|_T \in \Ht^1_{sym}(T)$ for all $T\in \triang$, $\tau_{nn}|_F \in L^2(F)$ be a normal-normal continuous piecewise smooth tensor field. Note that this includes all finite element tensor fields $\taut_h$.
For $\vv_h \in \spaceV_h$, let $\vv_h = \IntV \zv + \nabla p_h$ be the decomposition from Lemma~\ref{lem:vas}. We have
\begin{equation}
	\bilB(\taut, \vv_h) = \bilB(\taut, \IntV \zv) + \bilB(\taut, \nabla p_h).
\end{equation}
We estimate the two parts separately. For the estimate concerning $\zv$, we first use that $\bilB(\cdot,\cdot)$ is continuous in the $\Lt^2$/broken $\Hv^1$ setting.
\begin{eqnarray}
	\bilB(\taut, \IntV \zv) &\leq& \left( \|\taut\|_{\Lt^2(\Omega)} + \left(\sum_{F \in \faces} h_F \|\tau_{nn}\|_{L^2(F)}^2\right)^{1/2} \right) \|\IntV \zv \|_{\Hv^1,h} \\
	&\leq& \|\taut\|_{\spaceSigma_h} \|\IntV \zv \|_{\Hv^1,h}
\end{eqnarray} 
Next, we utilize the Cl\'ement interpolation operator $\Clement$, which is continuous in $H^1$,
\begin{eqnarray}
	\bilB(\taut, \IntV \zv) &\leq&\|\taut\|_{\spaceSigma_h} \left(\|\IntV\zv - \Clement \zv \|_{\Hv^1,h} + \|\Clement \zv\|_{\Hv^1(\Omega)} \right)\\
	&\leq& c \|\taut\|_{\spaceSigma_h} \left( \|\IntV\zv - \Clement \zv \|_{\Hv^1,h} + \|\zv\|_{\Hv^1} \right).
\end{eqnarray}
By an inverse inequality for the finite element function $\IntV\zv - \Clement \zv $ we see
\begin{equation}
	\bilB(\taut, \IntV \zv)  \leq c \|\taut\|_{\spaceSigma_h} \left( h^{-1}\|\IntV\zv - \Clement \zv \|_{\Lv^2} + \|\zv\|_{\Hv^1} \right).
\end{equation}
Using interpolation error estimates for $\IntV$ and $\Clement$,
\begin{align}
	\|\IntV\zv - \Clement \zv \|_{\Lv^2} \leq \|\IntV \zv - \zv\|_{\Lv^2} + \|\Clement \zv - \zv\|_{\Lv^2} \leq ch \|\zv\|_{\Hv^1},	
\end{align}	
we arrive at
\begin{eqnarray}
	\bilB(\taut, \IntV \zv)  &\leq& c \|\taut\|_{\spaceSigma_h} \|\zv\|_{\Hv^1}\\
	&\leq& c \|\taut\|_{\spaceSigma_h}\|\vv_h\|_{\Hv(\opcurl)}. \label{eq:estimate1}
\end{eqnarray}

The estimate concerning $\nabla p_h$ follows directly,
\begin{eqnarray}
	\bilB(\taut, \nabla p_h) &\leq& \sup_{w_h \in \spaceW_h} \frac{\bilB(\taut, \nabla w_h)}{\|\nabla w_h\|_{\Lv^2}} \, \|\nabla p_h\|_{\Lv^2} \\
	& \leq & c \|\taut\|_{\spaceSigma_h}\, \|\vv_h\|_{\Lv^2(\Omega)}. \label{eq:estimate2}
\end{eqnarray}
Together, \eqref{eq:estimate1} and \eqref{eq:estimate2} lead to the desired continuity result.
\end{proof}

\subsection{Stability of the finite element problem}

According to \cite{BoffiBrezziFortin:13}, we need to provide stability of the bilinear forms with respect to the discrete norms, i.e. we need to show discrete kernel-coercivity of $\bilA(\cdot,\cdot)$ and an inf-sup condition for $\bilB(\cdot,\cdot).$

\begin{lemma}
The bilinear form $\bilA(\cdot,\cdot)$ is coercive on the discrete kernel $\opker(B_h) := \{\taut_h \in  \spaceSigma_h, \bilB(\taut_h, \vv_h) = 0 \ \forall \vv_h \in \spaceV_h\}$,
\begin{equation}
	\bilA(\taut_h,\taut_h) \geq c \|\taut_h\|^2_{\spaceSigma_h} \qquad \forall
	\taut_h \in  \opker(B_h).
\end{equation}
with generic constant $c$ independent of the mesh size.
\end{lemma}
\begin{proof}
Let $\taut_h \in \opker(B_h)$ be fixed. Since $\taut_h$ is a finite element function, the facet term in the discrete norm \eqref{eq:normsigmah} can be omitted,
\begin{equation}
	\|\taut_h\|_{\spaceSigma_h}^2 \leq c \left( \|\taut_h\|_{\Lt^2(\Omega)}^2 + \left(\sup_{w_h \in \spaceW_h} \frac{\bilB(\taut_h, \nabla w_h)}{\|\nabla w_h\|_{\Lv^2(\Omega)}} \right)^2 \right).
	\label{eq:norm}
\end{equation}
By definition of $\opker(B_h)$, we have
\begin{equation}
	\bilB(\taut_h, \vv_h) =  0.
\end{equation}
Since $\nabla \spaceW_h \subset \spaceV_h$, the supremum term in \eqref{eq:norm} vanishes,
\begin{equation}
	\|\taut_h\|_{\spaceSigma_h}^2 \leq c  \|\taut_h\|_{\Lt^2(\Omega)}^2 .
\end{equation}
Hence, coercivity is implied by coercivity of the compliance tensor,
\begin{equation}
	\bilA(\taut_h, \taut_h) \geq \lambda_{\min}(\Atel) \|\taut_h\|_{\Lt^2(\Omega)}^2 \geq c \lambda_{\min}(\Atel) \|\taut_h\|_{\spaceSigma_h}^2.
\end{equation}
\end{proof}

\begin{lemma}
The bilinear form $\bilB(\cdot,\cdot)$ is inf-sup stable on $\spaceSigma_h \times \spaceV_h$, for $\vv_h \in \spaceV_h$ there exists some $\sigmat_h \in \spaceSigma_h$ such that
\begin{equation}
	\bilB(\sigmat_h, \uv_h) \geq c \|\sigmat_h\|_{\spaceSigma_h}\|\uv_h\|_{\Hv(\opcurl)}.
\end{equation}
\end{lemma}
\begin{proof}
The proof is very similar to the one of Lemma~\ref{lem:binfsup} in the infinite dimensional setting. Nevertheless, we provide it in detail here.

Let $\vv_h \in \spaceV_h$ be fixed. According to the finite element theory using $\Lt^2$ and broken $\Hv^1$ norms, there exists a unique pair $(\taut_h, \uv_h) \in \spaceSigma_h \times \spaceV_h$ satisfying
\begin{align}
	&\bilA(\sigmat_h, \taut_h) + \bilB(\taut_h, \uv_h) = 0 &&\forall \taut_h \in \spaceSigma_h,\\
	&\bilB(\sigmat_h, \tilde\vv_h) = \int_\Omega \left( \vv_h \cdot \tilde\vv_h + \opcurl \vv_h \cdot \opcurl \tilde\vv_h\right) \,d\xv && \forall \tilde\vv_h \in \spaceV_h. \label{eq:defsigma2}
\end{align}
Moreover, we have the stability estimate
\begin{equation}
	\|\sigmat_h\|_{\Lt^2(\Omega)} \leq c\sup_{\tilde\vv_h \in \spaceV_h} \frac{\int_\Omega \left( \vv_h \cdot \tilde\vv_h + \opcurl \vv_h \cdot \opcurl \tilde\vv_h\right) \,d\xv}{\|\tilde\vv_h\|_{\Hv^1,h}}.
\end{equation}
Since $\|\tilde\vv_h\|_{\Hv^1,h} \geq c \|\tilde\vv_h\|_{\Hv(\opcurl)}$, we deduce
\begin{equation}
	\|\sigmat_h\|_{\Lt^2(\Omega)} \leq c\|\vv_h\|_{\Hv(\opcurl)}. \label{eq:1}
\end{equation}
Testing the second equation \eqref{eq:defsigma2} with a gradient function $\nabla w_h$, we see using that $\opcurl \nabla w_h = 0$,
\begin{equation}
\bilB(\sigmat_h, \nabla w_h) = \int_\Omega \vv_h \cdot \nabla w_h\,d\xv.
\end{equation}
Hence, we deduce
\begin{equation}
\sup_{w_h \in \spaceW_h} \frac{\bilB(\sigmat_h,\nabla w_h)}{\|\nabla w_h\|_{\Lv^2}} = 
\sup_{w_h \in \spaceW_h} \frac{\int_\Omega \vv_h \cdot \nabla w_h \,d\xv}{\|\nabla w_h\|_{\Lv^2}}
\leq \|\vv_h\|_{\Lv^2}. \label{eq:2}
\end{equation}
Adding up squared \eqref{eq:1} and \eqref{eq:2}, we can bound $\|\sigmat_h\|_{\spaceSigma_h}$ by $ \|\vv_h\|_{\Hv(\opcurl)}$,
\begin{equation}
	\|\sigmat_h\|_{\spaceSigma_h} \leq c \|\vv_h\|_{\Hv(\opcurl)}.
\end{equation}
Finally using $\tilde\vv_h = \vv_h$ as a test function in \eqref{eq:defsigma2}, we arrive at the desired result
\begin{equation}
	\bilB(\sigmat_h,\vv_h) = \|\vv_h\|_{\Hv(\opcurl)}^2 \geq c \|\sigmat_h\|_{\spaceSigma_h} \|\vv_h\|_{\Hv(\opcurl)}.
\end{equation}

\end{proof}

\subsection{Error estimates}

Since the finite element method is slightly nonconforming, $\spaceSigma_h \not \subset \spaceSigma$, the error cannot be bounded directly using the theory from \cite{BoffiBrezziFortin:13}. Instead, we rely on techniques from Strang's second lemma, where consistency and interpolation error bound the approximation error. 

\begin{theorem}
Let $(\sigmat,\uv) \in \spaceSigma  \times \spaceV$ be the solution to the elasticity problem \eqref{eq:spp1}, \eqref{eq:spp2}. Let $(\sigmat_h, \uv_h) \in \spaceSigma_h \times \spaceV_h$ be the finite element solution from \eqref{eq:spph1}, \eqref{eq:spph2}. Suppose $\sigmat, \uv$ be sufficiently smooth, then we have the error bound for $1 \leq s \leq k$
\begin{eqnarray}
	\lefteqn{\|\sigmat - \sigmat_h\|_{\spaceSigma_h} + \|\uv - \uv_h\|_{\Hv(\opcurl)}} \\
	&\leq& c \left( \sum_{T \in \triang} h_T^{2s} \left(
		\|\uv\|_{\Hv^{s+1}(T)}^2 + \|\sigmat\|_{\Ht^s(T)}^2 \right) \right)^{1/2}.
\end{eqnarray}
\end{theorem}

\begin{proof}
We divide the approximation error into two parts, the interpolation error and a consistency term. To this end, we add and subtract the interpolants $\IntSigma \sigmat$, $\IntV \uv$ and use the triangle inequality,
\begin{align}
\|\sigmat - \sigmat_h\|_{\spaceSigma_h} + &\|\uv - \uv_h\|_{\Hv(\opcurl)} \leq\\
& (\|\sigmat - \IntSigma \sigmat\|_{\spaceSigma_h} + \|\uv - \IntV\uv\|_{\Hv(\opcurl)}) +\label{eq:interperror}\\
& (\|\IntSigma\sigmat - \sigmat_h\|_{\spaceSigma_h} + \|\IntV \uv - \uv_h\|_{\Hv(\opcurl)}) \label{eq:consistencyerror}
\end{align}

We refer to the terms in \eqref{eq:interperror} as interpolation error, while the terms in \eqref{eq:consistencyerror} are referred to as consistency error. 
We first elaborate on the consistency error, which can itself be bounded by the interpolation error:

Due to the discrete stability, we have
\begin{align}
c& (\|\IntSigma\sigmat - \sigmat_h\|_{\spaceSigma_h} + \|\IntV \uv - \uv_h\|_{\Hv(\opcurl)}) \\
&\leq \sup_{\taut_h \in \spaceSigma_h} \frac{\bilA(\IntSigma\sigmat - \sigmat_h, \taut_h) + \bilB(\taut_h, \IntV\uv - \uv_h)}{\|\taut_h\|_{\spaceSigma_h}} + \sup_{\vv_h \in \spaceV_h} \frac{\bilB(\IntSigma\sigmat - \sigmat_h, \vv_h)}{\|\vv_h\|_{\Hv(\opcurl)}}\\
&= \sup_{\taut_h \in \spaceSigma_h} \frac{\bilA(\IntSigma\sigmat, \taut_h)  + \bilB(\taut_h, \IntV\uv)}{\|\taut_h\|_{\spaceSigma_h}} + \sup_{\vv_h \in \spaceV_h} \frac{\bilB(\IntSigma\sigmat - \sigmat_h, \vv_h)}{ \|\vv_h\|_{\Hv(\opcurl)}} \label{eq:consistencyerrorII}
\end{align}

We proceed for the first term, adding and subtracting $\int_\Omega \Atel \sigmat: \taut_h\, d\xv$ and using that the solution $\uv$ is sufficiently smooth to have $\epst(\uv) = \Atel \sigmat$ in $\Lt^2$,
\begin{eqnarray}
\lefteqn{\bilA(\IntSigma\sigmat, \taut_h)  + \bilB(\taut_h, \IntV\uv)} \\
&=& \int_{\Omega} \Atel(\IntSigma\sigmat-\sigmat):\taut_h\, d\xv  +\\
&& \sum_{T \in \triang} \left( \int_T (\Atel\sigmat - \epst(\IntV\uv)):\taut_h\,d\xv + \int_{\partial T} (\IntV \uv)_n \tau_{h,nn}\, ds \right)\\
&=& \int_{\Omega} \Atel(\IntSigma\sigmat-\sigmat):\taut_h\, d\xv  +\\
&& \sum_{T \in \triang} \left(\int_T (\epst(\uv)-\epst(\IntV\uv)):\taut_h\,d\xv - \int_{\partial T} (u_n - (\IntV \uv)_n) \tau_{h,nn}\, ds\right) \label{eq:94}\\
&\leq& c (\|\IntSigma\sigmat - \sigmat\|_{\Lt^2(\Omega)}\|\taut_h\|_{\Lt^2(\Omega)} +
	\|\uv - \IntV\uv\|_{\Hv^1(\Omega),h} \|\taut_h\|_{\Lt^2(\Omega)}  ).
\end{eqnarray}
In line \eqref{eq:94}, we used that the solution $\uv$ is continuous, and thus $u_n$ can be added to the surface integrals as a vanishing jump term.
Thus, we have reduced the first term of the consistency error \eqref{eq:consistencyerrorII} to the interpolation error,
\begin{eqnarray}
\lefteqn{\sup_{\taut_h \in \spaceSigma_h} \frac{\bilA(\IntSigma\sigmat, \taut_h)  + \bilB(\taut_h, \IntV\uv)}{\|\taut_h\|_{\spaceSigma_h}}}\\
 &\leq& c \left(\|\IntSigma\sigmat - \sigmat\|_{\Lt^2(\Omega)} + \|\uv - \IntV\uv\|_{\Hv^1,h} \right)\\
&\leq& c \left( \sum_{T \in \triang} h_T^{2s} \left(
		\|\uv\|_{\Hv^{s+1}(T)}^2 + \|\sigmat\|_{\Ht^s(T)}^2 \right) \right)^{1/2}.
\end{eqnarray}

We proceed to the second term in \eqref{eq:consistencyerrorII}.
We know that both $\sigmat$ and $\sigmat_h$ satisfy the equilibrium condition \eqref{eq:spp2} with test function $\vv_h \in \spaceV_h \subset \spaceV$, and thus $\bilB(\sigmat,\vv_h) = \bilB(\sigmat_h, \vv_h)$.
 We deduce
\begin{equation}
\bilB(\IntSigma\sigmat -\sigmat_h, \vv_h) = \bilB(\IntSigma\sigmat -\sigmat, \vv_h) = \langle \opdivv(\IntSigma\sigmat - \sigmat), \vv_h\rangle.
\end{equation}
We use Lemma~\ref{lemma:bcont} to show that $\langle \opdivv(\IntSigma\sigmat - \sigmat), \vv_h\rangle$ is bounded with respect to the discrete norms, as we assume $\sigmat$ to be piecewise smooth normal-normal continuous,
\begin{equation}
\langle \opdivv(\IntSigma\sigmat - \sigmat), \vv_h\rangle \leq c \|\IntSigma\sigmat - \sigmat\|_{\spaceSigma_h}\|\vv_h\|_{\Hv(\opcurl)}.
\end{equation}
Thus, the second term of the consistency error \eqref{eq:consistencyerrorII} can be bounded by
\begin{align}
\sup_{\vv_h \in \spaceV_h} \frac{\bilB(\IntSigma\sigmat -\sigmat_h, \vv_h)}{ \|\vv_h\|_{\Hv(\opcurl)}} & \leq c \|\IntSigma\sigmat - \sigmat\|_{\spaceSigma_h} \\
&\leq c  \left( \sum_{T \in \triang} h_T^{2s} \|\sigmat\|_{\Ht^s(\Omega)}^2 \right)^{1/2} .
\end{align}

Together with the interpolation error estimates we arrive at the desired results.
\end{proof}


\bibliographystyle{plain}      
\bibliography{TDNNS_NaturalNorms_bib}   

\end{document}